\documentclass[a4paper,11pt, reqno]{amsart}
\usepackage[utf8]{inputenc}

\usepackage{amssymb,amsmath,amsfonts,amsthm}

\usepackage{a4wide}
\usepackage[pdfpagelabels,hypertexnames=true,
            plainpages=false,
            naturalnames=false]{hyperref}
\usepackage{color}
\definecolor{darkblue}{rgb}{0,0.1,0.5}
\hypersetup{colorlinks,
            linkcolor=darkblue,
            anchorcolor=darkblue,
            citecolor=darkblue}

\numberwithin{equation}{section}

\usepackage{txfonts}

\usepackage[capitalise]{cleveref}
\usepackage{natbib}

\newtheorem{theorem}{Theorem}
\newtheorem{lemma}{Lemma}

\newtheorem{corollary}{Corollary}

\theoremstyle{remark}

\newcommand{\dd}{\mathrm{d}}
\newcommand{\ee}{\mathrm{e}}

\renewcommand{\leq}{\leqslant}
\renewcommand{\geq}{\geqslant}

\title[Limiting distribution of the maximal distance between random points on a circle]{Limiting distribution of the maximal distance between random points on a circle: A moments approach}%
\author{Eckhard Schlemm}%
\address{Wolfson College, University of Cambridge}
\curraddr{UCL Medical School, University College London}
\email{eckhard.schlemm@cantab.net}


\begin{document}

\begin{abstract}
Motivated by the problem of computing the distribution of the largest distance $d_{\max}$ between $n$ random points on a circle we derive an explicit formula for the moments of the maximal component of a random vector following a Dirichlet distribution with concentration parameters $(1,\ldots,1)$. We use this result to give a new proof of the fact that the law of $n\,d_{\max}-\log n$ converges to a Gumbel distribution as $n$ tends to infinity.
\end{abstract}

\maketitle

\section{Introduction}
For a positive integer $n$, we denote by $X_1,\ldots,X_n$ a collection of $n$ independent, standard uniform random variables, which we interpret as locations of points on a circle with perimeter one. By $d_i$, $i=1,\ldots,n$, we denote the distances (in arc length) between adjacent points, that is $d_i=X_{i+1}-X_i$, where the index is taken modulo $n$. Alternatively, the distances $d_i$ can be interpreted as the lengths of the pieces of a randomly broken stick of length one \cite{holst1980lengths}. A detailed understanding of their properties is of importance in some aspects of non-parametric statistics \cite{wilks1962mathematical}. 
The set of distances is also interesting from a purely probabilistic point of view because the smallest, typical, and largest distances show quite markedly different behaviour as the number of points tend to infinity. It is known \cite[Problem 6.4.2]{david2003order} that the expected size of the $k$\textsuperscript{th}-largest gap, $k=1,\ldots,n$, is given by $n^{-1}\sum_{j=k}^n{1/j}$. In particular, the smallest gap $d_{\min}$ is of order $1/n^2$, whereas the largest gap $d_{\max}$ is of order $H_n/n\sim\log n/n$, where $H_n=\sum_{j=1}^n1/j$ denotes the $n$\textsuperscript{th} harmonic number, and $a_n\sim b_n$ if and only if $a_n/b_n\to	 1$. An easy calculation shows that $n^2d_{\min}$ converges in distribution to an exponential random variable with parameter one. In this short note, we are concerned with the limiting distribution of a suitably scaled and centred version of $d_{\max}$. Using the observation that the $n$-tupel $(d_i)$ of distances follows a Dirichlet distribution with parameters $(1,\ldots,1)$ it can be deduced from \cite[Corollary 3.1.]{bose2008maxima} that $n\,d_{\max}-\log n$ converges in law to a Gumbel distribution. 

In the following we provide an alternative, combinatorical proof of that result, bypassing arguments from extreme value theory in the spirit of \cite{gnedenko1943sur} and \cite{leadbetter1983extremes}; we derive, for the first time, an explicit formula for the moments of a $\operatorname{Dir}(1,\ldots,1)$ distribution and compute their limits as $n$ tends to infinity. This allows us to identify the limiting distribution in \cref{section-limit}.

\section{Computation of the moments of $d_{\max}$}
By \cite[Eq. (6.4.4)]{david2003order}, the distribution function of the largest gap is given by
\begin{equation*}
\mathbb{P}\left(d_{\max}\leq x\right) = 1 - n(1-x)^{n-1} + \binom{n}{2}(1-2x)^{n-1}	- \ldots +(-1)^k\binom{n}{k}(1-kx)^{n-1} + \ldots,
\end{equation*}
where the sum continues as long as $kx\leq 1$. In particular, after differentiating with respect to $x$ and observing that $d_{\max}\geq 1/n$, the $m$\textsuperscript{th} moment of the largest spacing is given by
\begin{align}
\begin{split}
\label{eq-moment-mn}
\mathbb{E}\left[\left(d_{\max}\right)^m\right] =& (n-1)\sum_{\nu=1}^{n-1}{ \int_{1/(\nu+1)}^{1/\nu}{\sum_{k=1}^\nu{\binom{n}{k}(-1)^{k+1} k x^m (1-kx)^{n-2}}\dd x}}.
\end{split}
\end{align}
In the following we evaluate this expression in closed form. Since the order in which the summations and integration are carried out is inconsequential, we consider the integrals in \cref{eq-moment-mn} first.
\begin{lemma}
\label{lemma-x-integral}
For positive integers $n$, $k<n$, $\nu<n$ and $m$, the following holds.
\begin{equation}
\label{eq-x-integral}
\int_{1/(\nu+1)}^{1/\nu}{x^m (1-kx)^{n-2}\dd x} = \sum_{\mu=0}^m{\frac{1}{k^\mu} \frac{m!(n-2)!}{(m-\mu)!(n+\mu-2)!}T_{n+\mu,k,\nu ,m-\mu}},
\end{equation}
where
\begin{equation}
T_{n+\mu,k,\nu,m-\mu}=\frac{1}{k(n+\mu-1)}\left(\frac{(\nu+1-k)^{n+\mu-1}}{(\nu+1)^{n+m-1}} - \frac{(\nu-k)^{n+\mu-1}}{\nu^{n+m-1}}\right).
\end{equation}
\end{lemma}
\begin{proof}
The result is obtained by $m$-fold integration by parts.
\end{proof}
After changing the order of summations, this result can be used to perform the $\nu$-sum in \cref{eq-moment-mn}.
\begin{lemma}
\label{lemma-nu-sum}
For positive integers $n$, $k<n$ and $m$, the following holds.
\begin{equation}
\sum_{\nu=k}^{n-1}{\int_{1/(\nu+1)}^{1/\nu}{x^m (1-kx)^{n-2}\dd x}} = \sum_{\mu=0}^m{\frac{1}{k^{\mu +1}}\frac{m!(n-2)!}{(m-\mu)!(n+\mu-1)!}\frac{(n-k)^{n+\mu-1}}{n^{n+m-1}}}.
\end{equation}
\end{lemma}
\begin{proof}
After plugging in \cref{eq-x-integral} and interchanging the order of summation the sum over $k$ is seen to be telescoping, which gives the result.
\end{proof}

We also need the following binomial identities whose easy proofs are left to the reader.
\begin{lemma}
\label{lemma-binom-ident}
For positive integers $n$, $m$ and $s\leq m$, the following identities hold.
\begin{align}
\label{eq-binom-ident1}\sum_{k=1}^n{\binom{n}{k}(-1)^{k+1}\frac{k}{n+1-k}}=&(-1)^{n+1};\\
\label{eq-binom-ident2}\sum_{k=1}^n{k^s(-1)^{k+1} \binom{n}{k}} =& 0;\\
\label{eq-binom-ident3}\sum_{\mu=s}^m{\frac{(-1)^\mu}{(m-\mu)!(n+\mu-1)!}\binom{n+\mu-1}{\mu-s}}=&\delta_{s,m}\frac{(-1)^m}{(n+m-1)!}.
\end{align}
\end{lemma}
The following result records a link between raw moments and cumulants of a random variable and is used repeatedly in the sequel. 
\begin{lemma}
\label{lemma-partition-coefficient}
For a positive integer $m$ and real numbers $x_1,\ldots,x_m$, the quantity
\begin{equation}
\label{eq-partition-gen}
\sum_{\stackrel{r_1+2r_2+\ldots+m r_m=m}{r_i\in\mathbb{N}_0}}{\frac{m!}{r_1!1^{r_1}\cdots r_m!m^{r_m}}x_1^{r_1}\cdots x_m^{r_m}}
\end{equation}
can be interpreted as $m![y^m]\exp\left\{\sum_{r=1}^m{x_r y^r/r}\right\}$, where $[y^m]f(y)$ denotes the coefficient of $y^m$ in the formal power series $f(y)$. In particular, the $m$\textsuperscript{th} cumulant of a random variable with $m$\textsuperscript{th} raw moment given by \cref{eq-partition-gen} is equal to $\kappa_m=(m-1)!\,x_m$.
\end{lemma}
\begin{proof}
The first claim is proved by rearranging terms. It implies that the moment generating function of a random variable with raw moments $\mu_m'$ given by \cref{eq-partition-gen} is
\begin{equation*}
M(t) = 1 + \mu'_1t+\frac{\mu'_2}{2!}t^2+\frac{\mu'_3}{3!}t^3+\ldots=\exp\left\{\sum_{r=1}^\infty{\frac{x_r}{r}t^r}\right\}.
\end{equation*}
Its cumulants are thus easily computed from the cumulant generating function $K(t)=\log M(t)$ as
\begin{equation*}
\kappa_{m} = \left.\frac{\dd ^m}{\dd t^m}K(t)\right|_{t=0}=\left.\frac{\dd ^m}{\dd t^m}\sum_{r=1}^\infty{\frac{x_r}{r}y^r}\right|_{t=0}=(m-1)!\,x_m.\qedhere
\end{equation*}
\end{proof}

For the statement of the following auxiliary result, which we could not find proved in the literature, we introduce the notation
\begin{equation*}
\mathcal{H}_{n,s}\coloneqq\frac{1}{s!}\sum_{\mathbf{r}(s)}{\binom{s}{\mathbf{r}}H_{n,1}^{r_1}\cdots H_{n,s}^{r_s}}\coloneqq \sum_{\stackrel{r_1+2r_2+\ldots+s r_s=s}{r_1\in\mathbb{N}_0}}{\frac{1}{r_1!1^{r_1}\cdots r_s!s^{r_s}}H_{n,1}^{r_1}\cdots H_{n,s}^{r_s}},
\end{equation*}
where $\sum_{\mathbf{r}(s)}$ denotes a sum over integer partitions of $s$ and $H_{n,r}=\sum_{j=1}^n{1/j^r}$ is the $n$\textsuperscript{th} harmonic number of order $r$. The numbers $\mathcal{H}_{n,s}$, whose expression is intimately related to the Bell polynomials \cite{bell1927partition}, figure prominently in our expression for the moments of $d_{\max}$.
\begin{lemma}
\label{lemma-partition}
For any positive integers $n$ and $s$ the following holds.
\begin{equation}
\label{eq-partition}
\sum_{k=1}^n{k^{-s}\binom{n}{k}(-1)^{k+1}} = \mathcal{H}_{n,s}.
\end{equation}
\end{lemma}
\begin{proof}
The proof proceeds by induction on $n$. By \cref{lemma-partition-coefficient} it suffices to show that the left side of \cref{eq-partition} equals the coefficient of $y^s$ in $\exp\left\{\sum_{r=1}^s{H_{n,r}y^r/r}\right\}$. The claim is true for $n=1$, when both sides equal one. For the induction step we assume the validity of the statement up to $n$ and compute
\begin{align*}
[y^s]\exp\left[\sum_{r=1}^s{\frac{H_{n+1,r}}{r}y^r}\right] =& [y^s]\exp\left\{\sum_{r=1}^s{\frac{H_{n,r}}{r}y^r}+\sum_{r=1}^\infty{\frac{1}{r}\left(\frac{y}{n+1}\right)^r}\right\}\\
  =& \sum_{\sigma=0}^s\left([y^\sigma]\exp\left\{\sum_{r=1}^s{\frac{H_{n,r}}{r}y^r}\right\}\right)\left([y^{s-\sigma}]\frac{n+1}{n+1-y}\right)\\
  =& \sum_{\sigma=0}^s{\sum_{k=1}^n{k^{-\sigma}\binom{n}{k}(-1)^{k+1}}\frac{1}{(n+1)^{s-\sigma}}}\\
  =& \sum_{k=1}^n{\binom{n}{k}(-1)^{k+1}\left[k^{-s}\frac{n+1}{n+1-k}-(n+1)^{-s}\frac{k}{n+1-k}\right]}\\
  =&\sum_{k=1}^{n+1}{k^{-s}\binom{n+1}{k}(-1)^{k+1}}.\qedhere
\end{align*}
To obtain the last line we have used the identity \labelcref{eq-binom-ident1}.
\end{proof}

The main result of this section can now be proved.
\begin{theorem}
\label{theorem-moments-finite}
For positive integers $n$ and $m$, the $m$\textsuperscript{th} moment of $d_{\max}$ is given by
\begin{equation}
\mathbb{E}\left[\left(d_{\max}\right)^m\right] = \frac{(n-1)!}{(n+m-1!)}\sum_{\stackrel{r_1+2r_2+\ldots+m r_m=m}{r_i\in\mathbb{N}_0}}{\frac{m!}{r_1!\,1^{r_1}\,r_2!\,2^{r_2}\cdots r_m!\,m^{r_m}} H_{n,1}^{r_1}H_{n,2}^{r_2}\cdots H_{n,m}^{r_m}}.
\end{equation}
\end{theorem}
\begin{proof}
Combining \cref{eq-moment-mn,lemma-nu-sum} we can write
\begin{align*}
\mathbb{E}\left[\left(d_{\max}\right)^m\right] =& (n-1)\sum_{k=1}^{n-1}{\binom{n}{k}(-1)^{k+1}\sum_{\mu=0}^m{\frac{1}{k^\mu}\frac{m!(n-2)!}{(m-\mu)!(n+\mu-1)!}\frac{(n-k)^{n+\mu-1}}{n^{n+m-1}}}}.
\end{align*}
We use the binomial theorem to expand $(n-k)^{n+\mu-1}$ and change the order of summation to obtain
\begin{align*}
 \mathbb{E}\left[\left(d_{\max}\right)^m\right] =&\frac{(n-1)!}{n^m} \sum_{\mu =0}^m {\frac{m!}{(m-\mu )! (n+\mu-1)!}\sum _{s=0}^{n+\mu-1} (-1)^s n^{\mu-s} \binom{n+\mu -1}{s}\sum _{k=1}^{n-1}{k^{s-\mu}\binom{n}{k}(-1)^{k+1}}}.
\end{align*}
Splitting the sum according to the sign of the exponent $s-\mu$ of $k$, adjusting the summation index $s$, and using \cref{lemma-partition} as well as \cref{eq-binom-ident2}, we obtain
\begin{align*}
\mathbb{E}\left[\left(d_{\max}\right)^m\right] =& \frac{(n-1)!}{n^m} \sum_{\mu =0}^m \frac{m!}{(m-\mu )! (n+\mu-1)!}\left[\sum _{s=1}^\mu (-1)^{\mu-s} n^s \binom{n+\mu-1}{\mu-s}\sum _{k=1}^{n-1}{k^{-s}\binom{n}{k}(-1)^{k+1}}+\right.\\
  &\quad\left.+\sum _{s=0}^{n-1} (-1)^{\mu+s} n^{-s} \binom{n+\mu -1}{\mu+s}\sum _{k=1}^{n-1}{k^s(-1)^{k+1} \binom{n}{k}}\right]\\
   =&\frac{(n-1)!}{n^m} \sum_{\mu =0}^m \frac{m!}{(m-\mu )! (n+\mu-1)!}\left[(-1)^\mu\sum _{s=1}^\mu (-1)^s n^s \binom{n+\mu-1}{\mu-s}\mathcal{H}_{n,s}+\right.\\
  &\quad\left.+(-1)^{n+\mu}\sum _{s=1}^\mu (-1)^s \binom{n+\mu-1}{\mu-s}+(-1)^{n+\mu}\sum _{s=0}^{n-1} (-1)^s \binom{n+\mu -1}{\mu+s}\right].
\end{align*}
The last two terms are equal to $\pm(n+\mu-2)!/[(n-1)!(\mu-1)!]$ and cancel each other. Interchanging the order of summation and using \cref{eq-binom-ident3} we finally obtain
\begin{equation*}
\mathbb{E}\left[\left(d_{\max}\right)^m\right] = \frac{(n-1)!}{n^m} \sum_{s=1}^m{\delta_{s,m}\frac{m!(-1)^m}{(n+m-1)!}(-1)^s}n^s\mathcal{H}_{n,s}=\frac{(n-1)!m!}{(n+m-1)!}\mathcal{H}_{n,m}.\qedhere
\end{equation*}
\end{proof}


\section{Identification of the limiting distribution}
\label{section-limit}
In this section we leave the finite setting and explore the asymptotic behaviour of the largest spacing $d_{\max}$. We use $o(1)$ to denote a term that converges to zero as $n$ tends to infinity. The first lemma shows that simply rescaling $d_{\max}$ by the inverse of its expected size does not lead to an interesting limit.
\begin{lemma}
The sequence $\frac{n}{\log n} d_{\max}$ converges to one in distribution.
\end{lemma}
\begin{proof}
Using \cref{theorem-moments-finite} one sees that all moments converge to one. Since the Dirac mass at one is the only measure with all moments equal to one, the claim follows.
\end{proof}
Since $\mathbb{E}\left[\frac{n}{\log n} d_{\max}-1\right]=\left[\gamma+o(1)\right]/\log n$, it is natural to consider $\log n\left(\frac{n}{\log n}d_{\max}-1\right)$ next. This scaling turns out to be correct.
\begin{theorem}
\label{theorem-moments-limit}
The moments of the centred and scaled maximal distance $n\,d_{\max}-\log n$  satisfy
\begin{equation}
\label{eq-moments-limit}
\mathbb{E}\left[\left(n\,d_{\max}-\log n\right)^m\right]\xrightarrow{n\to\infty}\mu_m\coloneqq\sum_{\mathbf{r}(m)}{\binom{m}{\mathbf{r}} \gamma^{r_1}\zeta(2)^{r_2}\cdots \zeta(m)^{r_m}},
\end{equation}
where $\gamma$ is the Euler--Mascheroni constant and $\zeta(\cdot)$ denotes the Riemann $\zeta$-function. In particular, the $m$\textsuperscript{th} cumulant of $n\,d_{\max}-\log n$ converges to $\gamma$ for $m=1$, and to $(m-1)!\,\zeta(m)$ for $m\geq 2$.
\end{theorem}
\begin{proof}
The binomial theorem implies that the left side of \cref{eq-moments-limit} can be written as
\begin{equation*}
\sum_{k=0}^m{(\log n)^{m-k}n^k\binom{m}{k}(-1)^{m-k}\frac{(n-1)!}{(n+k-1)!}\sum_{\mathbf{r}}{\binom{k}{\mathbf{r}}H_{n,1}^{r_1}\cdots H_{n,k}^{r_k}}}.
\end{equation*}
Since $(n-1)!/(n+k-1)!\sim n^{-k}$, it follows that, in the limit, this sum can be interpreted as the $m$\textsuperscript{th} moment of $\tilde d_{\max}-\log n$, where $\tilde d_{\max}$ is a random variable with raw moments $\sum_{\mathbf{r}}{\binom{m}{\mathbf{r}}H_{n,1}^{r_1}\cdots H_{n,m}^{r_m}}$. For such a random variable, however, the cumulants are easily computed using \cref{lemma-partition-coefficient}, namely
\begin{equation*}
\kappa_{n,1}=H_{1,n}=\log n+\gamma +o(1),\quad\text{and}\quad \kappa_{n,m}=(m-1)!H_{n,m}=(m-1)![\zeta(m)+o(1)],\quad m\geq 2.
\end{equation*}
Using the fact that the first cumulant is shift-equivariant and that higher cumulants are shift-invariant we find that the cumulants of $\tilde d_{\max}-\log n$, and thus also the cumulants of $n\,d_{\max}-\log n$, converge to
\begin{equation*}
\kappa_1=\gamma,\quad\text{and}\quad \kappa_m=(m-1)!\,\zeta(m),\quad m\geq 2.
\end{equation*}
This convergence of cumulants implies convergence of moments to the limit given in the statement of the theorem because each moment is a continuous function of finitely many cumulants.
\end{proof}
We can now use this information to identify the limiting distribution of the maximal spacing between random points on a circle.
\begin{corollary}
The rescaled and centred maximal distance $n\,d_{\max}-\log n$ converges in law to a standard Gumbel distribution with location parameter zero and scale one.
\end{corollary}
\begin{proof}
It is well known that the limiting cumulants of $n\,d_{\max}-\log n$ obtained in \cref{theorem-moments-limit} are exactly those of a standard Gumbel distribution. It thus suffices to show that these cumulants, or equivalently the moments $(\mu_m)$ determine a unique probability distribution. This, however, follows from \cite[Theorem 3.3.11.]{durret2010probability} if it can be shown that the sequence $(\mu_{2m}^{1/2m}/2m)$ remains bounded. Following \cite{apostol1976introduction} to bound the number of partitions of $2m$ by $\exp\left\{c\sqrt{2m}\right\}$, where $c=2\sqrt{\zeta(2)}$, and further bounding $\gamma<2$, $\zeta(r)<2$, it follows that
\begin{equation*}
\frac{1}{2m}\mu_{2m}^{1/2m}<\frac{1}{2m}\left(\exp\left\{c\sqrt{2m}\right\}(2m)!2^{2m}\right)^{1/2m}\xrightarrow{m\to\infty}\frac{2}{\ee }<\infty.
\end{equation*}
This completes the proof of the corollary.
\end{proof}

\bibliographystyle{plainnat}

\end{document}